%% file: arxiv_v3.tex
\documentclass[9pt,reqno]{amsart}  
\input{packages}
\input{notations}
\input{bibliography_setup.tex}
\bibliography{bibbulk}

\numberwithin{equation}{section} 
\newtheorem{theorem}{Theorem}[section]
\newtheorem{assumption}[theorem]{Assumption}
\newtheorem{lemma}[theorem]{Lemma}
\newtheorem{proposition}[theorem]{Proposition}

\newtheorem{remark}[theorem]{Remark}
\newtheorem{conjecture}[theorem]{Conjecture}

\date{\today}
\author{Giorgio Cipolloni\(^{\dagger\ddagger}\) \and L\'aszl\'o Erd\H{o}s\(^{\dagger}\)}
\address{IST Austria, Am Campus 1, 3400 Klosterneuburg, Austria}
\author{Dominik Schr\"oder\(^{\ast}\)}
\address{Institute for Theoretical Studies, ETH Zurich, Clausiusstr.\ 47, 8092 Zurich, Switzerland}
\email{giorgio.cipolloni@ist.ac.at} 
\email{lerdos@ist.ac.at}
\email{dschroeder@ethz.ch}
\thanks{\(^\dagger\)Partially supported by ERC Advanced Grant No.~338804}
\thanks{\(^\ddagger\)This project has received funding from the European Union's Horizon 2020 research and innovation programme under the Marie Sk\l odowska-Curie Grant Agreement No. 665385.}
\thanks{\(^\ast\)Supported by Dr.\ Max R\"ossler, the Walter Haefner Foundation and the ETH Z\"urich Foundation}
\subjclass[2010]{60B20, 15B52} 
\keywords{Ginibre ensemble, Girko's Formula, Bulk universality}
\title{Towards the bulk universality of non-Hermitian random matrices}
\date{\today}
\begin{document} 
\thispagestyle{empty}

\begin{abstract}
    We consider the non-Hermitian analogue of the celebrated Wigner-Dyson-Mehta
    bulk universality phenomenon, i.e.\ that in the bulk the local eigenvalue statistics  
    of a large random matrix with independent, identically distributed centred entries
    are universal, in particular they asymptotically coincide with those of the Ginibre ensemble
    in the corresponding symmetry class. 
    In this paper we reduce this problem to understanding a certain microscopic regime
    for the Hermitized resolvent in Girko's formula by showing that all other regimes are
    negligible.
\end{abstract}

\maketitle

\section{Introduction}

Consider a large \(n\times n\) matrix \(X\) with independent, identically distributed \emph{(i.i.d.)}  centred entries
with variance \(n^{-1}\). According to the \emph{circular law}~\cite{MR1428519, MR773436, MR2908617, taovu2012}, the spectrum of \(X\) converges to the unit disk in the complex plane with uniform spectral density. The typical distance between nearby eigenvalues is \(n^{-1/2}\). We consider
the eigenvalue point process after rescaling it by a factor
of \(n^{1/2}\) around a fixed point \(z_0\in \C\), \(\abs{z_0}<1\). In case of the \emph{Ginibre ensemble}, i.e.\ if the entries of \(X\) are
Gaussian, all correlation
functions of this rescaled point process  can be computed explicitly, for both the real and the complex case, in the \(n\to \infty\) limit,   
see Remark~\ref{rem:explicit}. Beyond the Gaussian case no explicit formulas 
are available, but the outstanding conjecture asserts that 
the local eigenvalue statistics are given by exactly the same  formulas for essentially any distribution of the matrix elements. 
In this paper we make a step towards proving this conjecture in the bulk regime.  
The analogous universality result at the edge of the spectrum, \(\abs{z_0}=1\),
has been fully proven recently in~\cite{1908.00969} relying on supersymmetric methods to obtain a
lower tail estimate for the lowest singular value of \(X\)~\cite{1908.01653}.
Prior to our works, these universality conjectures have only been proven
under the restriction that the first four moments of the common distribution of the  matrix elements (almost) match
the first four moments of the standard Gaussian~\cite{taovu2015}. Matching the second moment amounts to a simple rescaling, but
the requirement of
matching any higher moments was an artefact of the proof. 

Local spectral universality questions have been motivated by Eugene Wigner's pioneering idea to model 
spectral statistics of complex quantum systems by those of simple random matrix ensembles
that respect  the basic symmetries but otherwise may not resemble at all to the initial quantum Hamiltonian.  The original 
Wigner-Dyson-Mehta (WDM) conjecture~\cite{MR0220494} concerned Hermitian random matrix ensembles, most prominently the
Wigner ensemble that is characterized by i.i.d.\ entries (up to the Hermitian symmetry). Since the resolution
of the WDM conjecture about ten years ago via the \emph{three-step strategy}
(see~\cite{bulletin, MR3699468} for an overview of the major steps and references), in the recent years many local spectral universality results 
have been obtained for  random matrix ensembles of increasing generality. However, apart from~\cite{taovu2015} 
and~\cite{1908.00969} all universality results have been restricted to Hermitian ensembles. 

The main reason why the three-step strategy has not yet been extended beyond the Hermitian case is the 
lack of  a good  analogue of the celebrated \emph{Dyson Brownian Motion} (DBM), a system of stochastic differential
equations for the eigenvalues under a natural matrix flow.  The  DBM is the essential core of the
three-step strategy; its  fast convergence to local 
equilibrium is the ultimate  mechanism behind  universality. This dynamical approach is extremely robust since
it not only detects universality but also induces it. Unfortunately, the non-Hermitian analogue of the DBM~\cite[Appendix A]{1801.01219}
involves overlaps of eigenvectors as well, making the rigorous analysis extremely complicated and currently beyond reach.

In the current approach, similarly  to our edge universality proof~\cite{1908.00969}, 
we circumvent the non-Hermitian DBM.\@ As standard in non-Hermitian spectral analysis, we use Girko's formula~\cite{MR773436} in the form given in~\cite{taovu2015}
that expresses linear eigenvalue statistics of \(X\) in terms
of  resolvents of a family of \(2n\times 2n\) Hermitian matrices
\begin{equation}
    \label{eq:linz1}
    H^z:=\begin{pmatrix}
        0 & X-z \\
        X^*-\overline{z} & 0
    \end{pmatrix}
\end{equation}
parametrized by \(z\in \C\).  This formula asserts that
\begin{equation}
    \label{girko}
    \sum_{\sigma\in \Spec(X)} f(\sigma) =\frac{1}{4\pi }\int_\C \Delta f(z) \log \abs{\det H^z}\, \diff^2z = -\frac{1}{4\pi} \int_{\C} \Delta f(z)\int_0^\infty \Im \Tr  G^z(\ii \eta)\diff\eta \diff^2z
\end{equation}
for any smooth, compactly supported test function \(f\), where \(G^z(w):= (H^z-w)^{-1}\) is the resolvent of \(H^z\).
The key point is that we are back to the Hermitian world and all tools and results 
developed for Hermitian ensembles in the last years are available. 

Utilizing  Girko's formula requires a very good understanding of the  resolvent  
of \(H^z\) along the imaginary axis for all \(\eta>0\).
The standard  \emph{local law} gives a computable deterministic approximation to \(\Im \Tr  G^z\)
with an error term of order \(1/\eta\) that is too crude for~\eqref{girko}, so we need a 
more accurate analysis.
\emph{A priori} all \(\eta\) regimes in~\eqref{girko}
may substantially contribute. The main result of this paper is to show that 
only the regime \(\eta\sim n^{-1}\) is relevant for the
local eigenvalue statistics  of \(X\). While this is not unexpected, the proof is  non-trivial. 
On very small scales \(\eta\ll n^{-1}\), there are no eigenvalues, hence \(\Im \Tr  G^z\) is negligible.
Above this microscopic scale, i.e.\ for \(\eta\gg n^{-1}\), we show that the trace of the resolvent \(   G^z\)  varies slowly in  \(z\),
hence there is an additional cancellation in~\eqref{girko} when 
\(\Im \Tr  G^z\) is  integrated against \(\Delta f(z)\) that  has zero integral. We exploit this cancellation 
by a first order Taylor expansion of the function \(z\to \Im \Tr  G^z\) and an auxilliary bound from~\cite{NonhermitianGGG}. 

This leaves the scale \(\eta\sim n^{-1}\) unexplored which is equivalent to
understanding a few small singular values of \(X-z\). 
We note that for a \emph{single} \(z\), the universality of the few smallest singular values of \(X-z\) was proven in~\cite{MR3916329} in the complex case. However, owing to the \(z\)-integration in~\eqref{girko},
one also needs the universality of the joint distribution the smallest  singular values
of \(X-z_1, X-z_2, \ldots, X-z_k\) for \emph{any finite collection} of \(z_i\) at distance \(n^{-1/2}\) from each other.
While this more general form of universality of the singular values is certainly expected to hold,
the proof of~\cite{MR3916329} currently cannot cover this generalization.

We remark that ideas  based solely on local laws and Green function comparison arguments
were sufficient for the edge proof in~\cite{1908.00969},
no Hermitian universality result was needed. The bulk regime is different, the necessary information 
on \(\Im \Tr  G^z(i\eta)\) for \(\eta\sim n^{-1}\) is apparently not accessible solely by these  methods. 
The classical Wigner-Dyson-Mehta universality for general Wigner matrices features the same distinction; all existing proofs
of sine-kernel universality in the \emph{bulk} spectrum requires some version of Dyson Brownian motion, 
while the Tracy-Widom universality for extreme eigenvalues at the \emph{edges}  of the Wigner semicircle law
can be proven  by carefully analyzing the Green function~\cite{MR3405746} (in fact even moment
method suffice~\cite{MR1727234}).

\subsection*{Notations and conventions}

We introduce some notations we use throughout the paper. For any \(k\in\N\) we use the notation \([k]:=\set{1,\ldots,k}\). We write \(\DD\) for the unit disk, \(\HC \) for the upper half-plane \(\HC :=\set{z\in\C\given\Im z>0}\), and for any \(z\in\C\) we use the notation \(\diff^2 z:= \frac{1}{2} \ii (\diff z\wedge \diff \overline{z})\) 
for the two dimensional volume form on \(\C\). For any \(2n\times 2n\) matrix \(A\) we use the notation \(\braket{ A}:= (2n)^{-1}\Tr  A\) to denote the normalized trace of \(A\). For positive quantities \(f,g\) we write \(f\lesssim g\) and \(f\sim g\) if \(f \le C g\) or \(c g\le f\le Cg\), respectively, for some constants \(c,C>0\) which depends only on the constants appearing in~\eqref{eq:boundmom}. We denote vectors by bold-faced lower case Roman letters \(\vx , \vy \in\C^k\), for some \(k\in\N\). Vector and matrix norms, \(\norm{ \vx }\) and \(\norm{ A}\), indicate the usual Euclidean norm and the corresponding induced matrix norm. Moreover, for a vector \(\vx \in\C^k\), we use the notation \(\diff \vx := \diff x_1\cdots \diff x_k\).

We will use the concept of ``with very high probability'' meaning that for any fixed \(D>0\) the probability of the event is bigger than \(1-n^{-D}\) if \(n\ge n_0(D)\). Moreover, we use the convention that \(\xi>0\) denotes an arbitrary small constant.

We use the convention that quantities without tilde refer to a general matrix with i.i.d.\ entries, whilst any quantity with tilde refers to the Ginibre ensemble, e.g.\ we use \(X\), \(\{\sigma_i\}_{i=1}^n\) to denote a non-Hermitian matrix with i.i.d.\ entries and its eigenvalues, respectively, and \(\widetilde{X}\), \(\{\widetilde{\sigma}_i\}_{i=1}^n\) to denote their Ginibre counterparts.

\section{Bulk universality conjecture}

We consider real or complex i.i.d.\ matrices \(X\), i.e.\ matrices whose entries are independent and identically distributed as \(x_{ab} \stackrel{d}{=} n^{-1/2}\chi\) for a (real or complex) random variable \(\chi\). We formulate the following assumption for  \(\chi\): 
\begin{assumption}\label{ass:a}
    We assume that \(\E \chi=0\) and \(\E\abs{\chi}^2=1\). In the complex case we also 
    assume \(\E\chi^2=0\) (this holds, for example, if \(\Re \chi\) and \(\Im \chi\) are i.i.d.).
    In addition, we assume the existence of high moments, i.e.\ that there exist constants \(C_p>0\) for each \(p\in\N\), such that
    \begin{equation}
        \label{eq:boundmom}
        \E \abs{\chi}^p\le C_p.
    \end{equation}
\end{assumption}

We denote the eigenvalues of \(X\) by \(\sigma_1,\ldots,\sigma_n\in\C\), and define the \emph{\(k\)-point correlation function \(p_k^{(n)}\)} of \(X\) implicitly as
\begin{equation}
    \label{eq:kpointfunc}
    \int_{\C^k}F(z_1,\ldots, z_k) p_k^{(n)}(z_1,\ldots, z_k)\,\diff^2 z_1\cdots\diff^2 z_k=\binom{n}{k}^{-1}\E  \sum_{i_1,\ldots, i_k} F(\sigma_{i_1}, \ldots, \sigma_{i_k}),
\end{equation}
for any smooth compactly supported test function \(F\colon\C^k \to \C\), with \(i_j\in \{1,\ldots, n\}\) for \(j\in\{1,\ldots,k\}\) all distinct. For the important special case when \(\chi\) follows a standard real or complex Gaussian distribution, we denote the \(k\)-point function of the \emph{Ginibre matrix} \(X\) by \(p_k^{(n,\Gin (\F))}\) for \(\F=\R,\C\). The \emph{circular law} implies that the \(1\)-point function converges 
\[\lim_{n\to\infty }p_1^{(n)}(z) = \frac{1}{\pi} \bm1(z\in\DD) = \frac{1}{\pi}\bm1(\abs{z}\le 1)\]
to the uniform distribution on the unit disk. On the scale \(n^{-1/2}\) of individual eigenvalues the scaling limit of the \(k\)-point function has been explicitly computed in the case of complex and real Ginibre matrices, \(X\sim\Gin (\R),\Gin (\C)\), i.e.\ for any fixed \(z_1,\ldots,z_k, w_1,\ldots,w_k\in\C\) there exist scaling limits \(p_{z_1,\ldots,z_k}^{(\infty)}=p_{z_1,\ldots,z_k}^{(\infty,\Gin (\F))}\) for \(\F=\R,\C\) such that
\[ 
\lim_{n\to\infty} p_k^{(n,\Gin (\F))}\Bigl(z_1+\frac{w_1}{n^{1/2}},\ldots,z_k+\frac{w_k}{n^{1/2}}\Bigr) = p_{z_1,\ldots,z_k}^{(\infty,\Gin (\F))}(w_1,\ldots,w_k).
\]

\begin{remark}\label{rem:explicit}
    The \(k\)-point correlation function \(p_{z_1,\ldots, z_k}^{(\infty,\Gin (\F))}\) of the Ginibre ensemble in both the complex and real cases \(\F=\C,\R\) is explicitly known; see~\cite{MR173726} and~\cite{MR0220494} for the complex case, and~\cite{MR2530159,MR1437734,17930739} for the real case, where the appearance of \(\sim n^{1/2}\) real eigenvalues causes a singularity in the density. In the complex case \(p_{z_1,\ldots, z_k}^{(\infty,\Gin (\C))}\) is determinantal, i.e.\ for any \(w_1,\ldots, w_k\in\C\) it holds
    \[p_{z_1,\ldots, z_k}^{(\infty,\Gin (\C))}(w_1,\ldots,w_k)= \det\left(K_{z_i,z_j}^{(\infty,\Gin (\C))}(w_i,w_j)\right)_{1\le i,j\le k}\]
    where for any complex numbers \(z_1\), \(z_2\), \(w_1\), \(w_2\) the kernel \(K_{z_1,z_2}^{(\infty,\Gin (\C))}(w_1,w_2)\) is defined by
    \begin{enumerate}[label=(\roman*)]
        \item For \(z_1\ne z_2\), \(K_{z_1,z_2}^{(\infty,\Gin (\C))}(w_1,w_2)=0\).
        \item For \(z_1=z_2\) and \(\abs{z_1}>1\), \(K_{z_1,z_2}^{(\infty,\Gin (\C))}(w_1,w_2)=0\).
        \item For \(z_1=z_2\) and \(\abs{z_1}<1\),
        \[
        K_{z_1,z_2}^{(\infty,\Gin (\C))}(w_1,w_2)=\frac{1}{\pi}e^{-\frac{\abs{w_1}^2}{2}-\frac{\abs{w_2}^2}{2}+w_1\overline{w_2}}.
        \]
        \item For \(z_1=z_2\) and \(\abs{z_1}=1\),
        \[
        K_{z_1,z_2}^{(\infty,\Gin (\C))}(w_1,w_2)=\frac{1}{2\pi}\left[ 1+ \erf\left( -\sqrt{2}(z_1\overline{w_2}+w_1\overline{z_2})\right) \right] e^{-\frac{\abs{w_1}^2}{2}-\frac{\abs{w_2}^2}{2}+w_1\overline{w_2}},
        \]
        where
        \[
        \erf(z):= \frac{2}{\sqrt{\pi}}\int_{\gamma_z} e^{-t^2} \diff t,
        \]
        for any \(z\in\C  \), with \(\gamma_z\) any contour from \(0\) to \(z\).
    \end{enumerate}
    For the corresponding much more involved formulas for \(p_k^{(\infty,\Gin (\R))}\) we defer the reader to~\cite{MR2530159}. 
\end{remark}

It is conjectured that \(p_{z_1,\ldots,z_k}^{(\infty,\Gin (\R,\C))}\) is universal (we recently proved this conjecture at the edge of the spectrum of \(X\) when all \(\abs{z_i}=1\)~\cite{1908.00969}).

\begin{conjecture}[Bulk universality]\label{theo:bulkuniv}
    Let \(X\) be an i.i.d.\ \(n\times n\) matrix with real or complex entries that satisfy Assumption~\ref{ass:a}.
    Then, for any fixed integer \(k\ge 1\), for any \(\tau>0\), for any complex spectral parameters \(z_1, \ldots, z_k\) such that \(\abs{z_j}\le 1-\tau\), \(j=1,\ldots,k\), and for any compactly supported smooth function \(F\colon\C^k \to \C\), we have the bound
    \begin{equation} 
        \label{eq:univ}
        \lim_{n\to +\infty}\int_{\C^k} F(\vw )\left[  p_k^{(n)} \left(\vz +\frac{\vw }{\sqrt{n}}\right)-p_{\vz }^{(\infty,\Gin (\F))}(\vw ) \right]\diff\vw = 0.
    \end{equation}
\end{conjecture}

Without loss of generality we may assume that the \(n\)-independent test function \(F\) is of the form
\[ F(w_1,\ldots,w_k)=f^{(1)}(w_1)\cdots f^{(k)}(w_k),\]
with \(f^{(1)},\ldots, f^{(k)}\) being smooth and compactly supported. Indeed, any smooth function \(F\) can be effectively approximated by its truncated Fourier series (multiplied by smooth cut-off function of product form); see also~\cite[Remark 3]{taovu2015}.
After a change of variables and using the inclusion-exclusion principle,~\eqref{eq:univ} amounts to proving that the eigenvalues \(\sigma_i\) of \(X\) and \(\wt \sigma_i\) of a comparison Ginibre ensemble \(\wt X\) satisfy
\begin{equation}\label{eq:impbound}
    \begin{split}
        &\E\prod_{j=1}^k\left(\frac{1}{n}\sum_{i=1}^n f^{(j)}_{z_j}(\sigma_i)- \frac{1}{\pi}\int_\DD  f^{(j)}_{z_j}(z)\diff^2 z\right) \\
        &\qquad = \E\prod_{j=1}^k \left(\frac{1}{n}\sum_{i=1}^n f^{(j)}_{z_j}(\widetilde{\sigma}_i)- \frac{1}{\pi}\int_\DD  f^{(j)}_{z_j}(z)\diff^2 z\right) + \landauO{ n^{-c(k)}},
    \end{split}
\end{equation}
where we introduced the rescaled test functions
\begin{equation}
    \label{eq:deftf}
    f^{(j)}_{z_j}(z):= n f^{(j)} (\sqrt{n}(z-z_j)), \quad z\in\C,
\end{equation}
and the implicit multiplicative constant in \(\landauO{\cdot}\) depends on the norms \(\norm{\Delta f^{(j)}}_1\), \(j=1,\ldots, k\). 

A possible approach to prove~\eqref{eq:impbound}
goes by analysing the Hermitization \(H^z\) of \(X-z\) defined 
in~\eqref{eq:linz1} since \(H^z\) and its resolvent \(G^z=G^z(\ii\eta)=(H^z-\ii\eta)^{-1}\) are related to the eigenvalues \(\sigma_i\) of \(X\) via Girko's Hermitization formula~\eqref{girko} 
and each factor in~\eqref{eq:impbound} can be written as
\begin{equation}
    \label{girko2}
    \frac{1}{n}\sum_{i=1}^n f^{(j)}_{z_j}(\sigma_i)- \frac{1}{\pi}\int_\DD  f^{(j)}_{z_j}(z)\diff^2 z
    = \frac{1}{2\pi}\int_\C   \Delta f^{(j)}_{z_j}(z)\int_0^\infty \braket{\Im G^z(\ii \eta)-\Im m^z(\ii \eta)} \diff \eta\diff^2 z,
\end{equation}
where \(m^z\) is the solution of the \emph{Dyson equation}~\eqref{eq:smde}, and we also used the identity  (cf.~\cite[Definition 2.3]{MR3770875})
\[\frac{1}{\pi}\int_\DD  f^{(j)}_{z_j}(z)\diff^2 z =  \frac{1}{2\pi}\int_\C   \Delta f^{(j)}_{z_j}(z)\int_0^\infty
\Im m^z(\ii \eta) \diff \eta\diff^2 z.\]
The contribution of the regime \(\eta\sim 1/n\) in~\eqref{girko2}
is given by \( \mathcal{I}_\epsilon(X,  f_{z_j}^{(j)}) \) where we define
\begin{equation}\label{eq:mainc}
    \mathcal{I}_\epsilon(X,  g_{z_0}) : = \frac{1}{2\pi}\int_\C   \Delta g_{z_0}(z)\int_{n^{-1-\epsilon}}^{n^{-1+\epsilon}}
    \braket{\Im G^z(\ii \eta)-\Im m^z(\ii \eta)} \diff \eta\diff^2 z
\end{equation}
and recall \(g_{z_0}(z):= n g (\sqrt{n}(z-z_0))\). Our main result is that for each factor in~\eqref{eq:impbound} the main contribution from Girko's Hermitization formula~\eqref{girko2} is given by \(\mathcal I_\epsilon\).
\begin{theorem}\label{pro:wcp}
    For fixed \(\tau>0\) and any sufficiently small \(\epsilon>0\) there exists a constant \(C=C_\tau>0\) such that for any compactly supported function \(g\) and any \(z_0\in \C\) with \(\abs{z_0}\le 1-\tau\) it holds that
    \begin{equation}\label{fI}
        \frac{1}{n}\sum_{i=1}^n g_{z_0}(\sigma_i)- \frac{1}{\pi}\int_\DD  g_{z_0}(z)\diff^2 z =  \mathcal{I}_\epsilon(X, g_{z_0}) + \mathcal{E}_\epsilon
    \end{equation}
    with an error \(\cE_\epsilon\) of size
    \begin{equation}\label{Eest}
     \E \abs{ \cE_\epsilon}\le C_\tau \norm{\Delta g}_1 n^{-\epsilon/4}.
    \end{equation}
\end{theorem}

As a consequence we can reduce Conjecture~\ref{theo:bulkuniv} to a conjecture about joint moments of \(\cI_\epsilon\) for arbitrarily small \(\epsilon>0\).
\begin{proposition}\label{prop reduction}
    Suppose there exists \(\epsilon>0\) such that for each \(k\in\N\) there exists \(c(k)>0\) such that 
    \begin{equation}
        \label{eq:c}
        \E \prod_{j=1}^k \mathcal{I}_\epsilon(X, f^{(j)}_{z_j}) =  \E \prod_{j=1}^k \mathcal{I}_\epsilon(\widetilde X, f^{(j)}_{z_j}) +\landauO*{n^{-c(k)}},
    \end{equation}
    for any collection of smooth compactly supported test functions \(f^{(j)}\), then Conjecture~\ref{theo:bulkuniv} holds true.
\end{proposition}
In a previous version of this paper we claimed a proof of~\eqref{eq:c} based upon the universality of singular values of \(X-z\) (see~\cite[Theorem 3.2]{MR3916329} in the complex case, and~\cite[Theorem 2.8]{2002.02438} in the real case). However, this result implies~\eqref{eq:c} only for \(k=1\); for general \(k\ge 2\) one would need a multi-\(z\) version of~\cite[Theorem 3.2]{MR3916329} and~\cite[Theorem 2.8]{2002.02438}, that is not yet available.

The main inputs for the proof of Theorem~\ref{pro:wcp} are the following two propositions. The first one is the optimal local law for \(G^z\) in Proposition~\ref{prop:locallaw}. It  asserts that
in the limit \(n\to +\infty\) the resolvent of \(H^z\) becomes deterministic and its limit can be found by solving the scalar equation
\begin{equation}
    \label{eq:smde}
    -\frac{1}{m^z}=w+m^z-\frac{\abs{z}^2}{w+m^z}, \quad m^z(w)\in\HC , \quad w\in \HC ,
\end{equation}
which is a special case of the \emph{matrix Dyson equation} (MDE), see e.g.~\cite{MR3916109}. 
On the imaginary axis \(m^z(\ii\eta)=\ii \Im m^z(\ii\eta)\). 
Then for \(\eta>0\) we define
\[
u=u^z(\ii\eta):=\frac{\Im m^z(\ii\eta)}{\eta+m^z(\ii\eta)}, \quad M=M^z(\ii\eta):=\left( \begin{matrix}
    m^z(\ii\eta) & -zu(\ii\eta) \\
    -\overline{z}u(\ii\eta) & m^z(\ii\eta) 
\end{matrix}\right).
\]
Moreover, 
\begin{equation}\label{Mbound}
    u^z(\ii\eta)\lesssim 1, \quad \norm{ M^z(\ii\eta)}\lesssim1,
\end{equation}
hold uniformly in \(z\) as long as \(\abs{z}\le 1-\tau\) for some fixed \(\tau>0\). 

\begin{proposition}[Local law for \(G^z\) on the imaginary axis~\cite{1907.13631}]\label{prop:locallaw}
    Let \(X\) be an i.i.d.\ \(n\times n\) matrix, whose entries satisfy Assumption~\ref{ass:a}, and let \(H^z\) as in~\eqref{eq:linz1}. Then for any deterministic vectors \(\vx , \vy \) and matrix \(R\), and any \(\xi>0\), \(\tau>0\) we have the bound
    \begin{equation}\label{entry}
        \abs*{ \braket{\vx, (G^z(\ii\eta)-M^z(\ii\eta))\vy}}  \le n^\xi \norm{\vx}\norm{\vy} \left( \frac{1}{\sqrt{n\eta}}+\frac{1}{n\eta}\right)
    \end{equation}
    \begin{equation}\label{ave}
        \abs*{\braket{ R(G^z(\ii\eta)-M^z(\ii\eta))}} \le \frac{n^\xi\norm{ R}}{n\eta}, 
    \end{equation}
    with very high probability, simultaneously in all \(\abs{z}\le1-\tau\) and all \(\eta>n^{-2}\),  as long as  \(n\) is sufficiently large, \(n\ge n_0\), where \(n_0\)
    is uniform in \(z\), it depends only on \(\tau,\xi \) and the control parameters in Assumption~\ref{ass:a}.
\end{proposition}

This proposition was proved in~\cite[Theorem 5.2]{1907.13631}; see also~\cite[Appendix A]{1908.00969}
to extend the result in~\cite{1907.13631} to hold simultaneously in all \(\abs{z}<1-\tau\) and \(\eta>n^{-2}\). The averaged local law in~\eqref{ave} and the entry-wise local law (choosing \(\vx \) and \(\vy \) being the coordinate vectors in~\eqref{entry}) have been proven earlier in~\cite[Theorem 5.2]{MR3770875} (see also~\cite[Theorem 3.4]{MR3230002} for \(R=I\)).
In~\cite[Theorem 3.1]{2002.02438} we extended the local law away from the imaginary axis.

The second input is a lower tail estimate on the lowest  singular value of \(X-z\) to control the very small \(\eta\ll n^{-1}\) regime in~\eqref{girko2}.

\begin{proposition}[Tail estimate for \(\lambda_1^z\)]\label{cor:ssvz} 
    Fix \(\tau>0\) and consider \(z\in\C\) with \(\abs{z}\le1-\tau\). Then for any \(L>0\) the smallest singular value \(\lambda_1^z\) of \(X-z\) satisfies
    \begin{equation}\label{eq:boundsmeg}
        \Prob\left( \lambda_1^z\le n^{-1-L}\right)\lesssim n^{-L/2}.
    \end{equation}
\end{proposition}
Proposition~\ref{cor:ssvz} follows from~\cite[Theorem 3.2]{TaoVusmooth}, and the crude upper bound \(\lVert X\rVert\lesssim 1\) with very high probability (e.g.\ see~\cite[Eq. (2.8)]{MR3770875}).
Alternatively,~\eqref{eq:boundsmeg} also follows by~\cite[Proposition 5.7]{MR3770875} (which is an adaptation of~\cite[Lemma 4.12]{MR2908617}), without recurring to the quite sophisticated proof of~\cite[Theorem 3.2]{TaoVusmooth}, under an additional 
very mild regularity  assumption, namely 
that there exist \(\alpha,\beta>0\) such that \(\chi\), the rescaled entry of the matrix \(X\), has a density \(g\colon\C\to [0,+\infty)\) satisfying
\begin{equation}\label{eq:reg}
    g\in L^{1+\alpha}(\C), \qquad \lVert g\rVert_{L^{1+\alpha}(\C)}\lesssim n^\beta.
\end{equation}
The bound in~\eqref{eq:boundsmeg} will be used twice in our  proof. First, in Lemma~\ref{lem:boundseg}, we use~\eqref{eq:boundsmeg} for some very  large \(L=l-1>0\) to ensure that with very high probability there are no singular values of \(X-z\) very close to zero, i.e.\ that 
\begin{equation}\label{eq:smass}
    \Prob\left(\lambda_1^z \le n^{-l} \right)\lesssim n^{-(l-1)/2}
\end{equation}
for any \(l>1\) uniformly in \(\abs{z}\le 1-\tau\). Second, we will use~\eqref{eq:boundsmeg} for \(L=\delta\), for some small \(\delta>0\) both in Lemma~\ref{lem:hpbound} and
Lemma~\ref{lem:boundseg} to control the  regime \(\lambda_1^z\in [n^{-l}, n^{-1-\epsilon}]\). We  remark that~\eqref{eq:boundsmeg} for some small \(L=\delta\)  can also be proven with the following argument
that neither relies on~\cite[Theorem 3.2]{TaoVusmooth} nor assumes~\eqref{eq:reg}.  First notice that~\cite[Eq. (4a)]{1908.01653}  proves~\eqref{eq:boundsmeg} with a small \(L\) for the smallest singular value of Ginibre matrices. 
Then we can combine this bound with~\cite[Theorem 3.2]{MR3916329} in the complex case and~\cite[Theorem 2.8]{2002.02438} in the real case, to ensure that the same bound holds for i.i.d.\ matrices \(X\) with arbitrary distribution for \(\chi\).

\section{Proof of Theorem~\ref{pro:wcp}}\label{sec:pr}
In this section we start with some \emph{a priori} bounds in Girko's formula and then we conclude it with the proof of Theorem~\ref{pro:wcp}. From now on we fix the scales
\begin{equation}\label{eq:defsc}
    \eta_0:= n^{-1-\epsilon}, \qquad \eta_1:= n^{-1+\epsilon}, \qquad T=n^{100},
\end{equation}
for some small fixed \(\epsilon>0\). We split the \(\eta\)-integration
in Girko's formula~\eqref{girko} for the rescaled test functions \(g_{z_0}(z)=ng(\sqrt{n}(z-z_0))\) as 
\begin{equation}  
    \label{eq:bigdiv}
    \begin{split} 
        \frac{1}{n}\sum_{i=1}^n g_{z_0}(\sigma_i)-  \frac{1}{\pi}\int_\DD  g_{z_0}(z)\diff^2 z &= \frac{1}{4\pi n}\int_\C   \Delta g_{z_0}(z)\log\abs*{\det (H^z-\ii T)}\diff^2 z \\
        &\quad -\frac{1}{2\pi}\int_\C   \Delta g_{z_0}(z)\int_0^{\eta_0} \braket{\Im G^z(\ii \eta)-\Im m^z(\ii \eta)} \diff \eta\diff^2 z \\
        &\quad -\frac{1}{2\pi}\int_\C   \Delta g_{z_0}(z)\int_{\eta_0}^{\eta_1} \braket{\Im G^z(\ii \eta)-\Im m^z(\ii \eta)} \diff \eta\diff^2 z \\
        &\quad -\frac{1}{2\pi}\int_\C   \Delta g_{z_0}(z)\int_{\eta_1}^{T} \braket{\Im G^z(\ii \eta)-\Im m^z(\ii \eta)} \diff \eta\diff^2 z \\
        &\quad +\frac{1}{2\pi}\int_\C   \Delta g_{z_0}(z)\int_T^{+\infty} \left( \Im m^z(\ii\eta)-\frac{1}{\eta+1}\right) \diff \eta\diff^2 z \\
        &=: J_T(g_{z_0})+I_0^{\eta_0}(g_{z_0})+I_{\eta_0}^{\eta_1}(g_{z_0})+I_{\eta_1}^T(g_{z_0})+I_T^{\infty}(g_{z_0}), 
    \end{split} 
\end{equation}
with \(\eta_0, \eta_1\), and \(T\) defined in~\eqref{eq:defsc}, so that \(\cI_\epsilon=I_{\eta_0}^{\eta_1}\). 

We split~\eqref{eq:bigdiv} into several integrals since the different regimes will be treated using different techniques. In particular, \(I_0^{\eta_0}\) is estimated using the lower tail bound in~\eqref{eq:boundsmeg} for the smallest eigenvalue (in absolute value) of \(H^z\); the integral \(I_{\eta_1}^T\) is estimated analysing the \(z\)-dependence of \(\braket{\Im G^z(\ii\eta)}\); finally, the integrals \(J_T\) and \(I_T^\infty\) are estimated by easy direct computations. This will show that the main contribution comes from the regime \(I_{\eta_0}^{\eta_1}\).

We start with giving \emph{a priori} bounds for the integrals in~\eqref{eq:bigdiv}. %
\begin{lemma}\label{lem:hpbound}
    It holds
    \begin{equation}\label{eq:easbound}
        \abs{J_T}\lesssim \frac{n^{1+\xi} \norm{\Delta g}_1}{T^2}, \quad \abs{I_0^{\eta_0}} +\abs{I_{\eta_0}^{\eta_1}}+\abs{I_{\eta_1}^T}\lesssim n^\xi\norm{\Delta g}_1,\quad \abs*{I_T^{\infty}}\lesssim \frac{n\norm{\Delta g}_1}{T},
    \end{equation}
    with very high probability for any \(\xi>0\).
\end{lemma}
\begin{proof}
    The bound for \(\abs{I_{\eta_0}^{\eta_1}}, \abs{I_{\eta_1}^T}\) follows by the local law for \(H^z\) in Proposition~\ref{prop:locallaw}. Using the bounds proven in~\cite[Proof of Theorem 2.5]{MR3770875}  we conclude the bounds of \(\abs{J_T}, \abs{I_T^\infty}\). By~\cite[Theorem 3.2]{TaoVusmooth} and a grid argument in \(z\) it follows that the bound for \(I_0^{\eta_0}\) in~\eqref{eq:easbound} holds on
     a very high probability set (see below~\eqref{eq:estsmeig1} for more details about this argument). Alternatively, 
     under the additional smoothness assumption~\eqref{eq:reg} the bound of \(\abs{I_0^{\eta_0}}\) also
     follows as in~\cite[Proof of Theorem 2.5]{MR3770875} directly without additional grid-argument.
\end{proof}

Next, for \(I_0^{\eta_0}\) and \(I_{\eta_1}^T\) we have improved bounds holding in expectation which allow to conclude Theorem~\ref{pro:wcp} and Proposition~\ref{prop reduction}.
\begin{lemma}\label{lem:boundseg}  
    For \(I_0^{\eta_0}\) with \(\eta_0=n^{-1-\epsilon}\) and for any \(\epsilon>0\), we have
    \[   \E \abs*{ I_0^{\eta_0}} \lesssim n^{-\epsilon/4} \norm{\Delta g}_1. \]
\end{lemma}
\begin{proposition}\label{prop I eta1 T}
    For \(I_{\eta_1}^T\) with \(\eta_1=n^{-1+\epsilon}\) and any \(\epsilon>0\) we have
    \[\E\abs{I_{\eta_1}^T} \lesssim n^{-\epsilon/4} \norm{\Delta g}_1. \]
\end{proposition}
\begin{proof}[Proof of Theorem~\ref{pro:wcp}]
    By~\eqref{eq:bigdiv}, Lemmata~\ref{lem:hpbound}--\ref{lem:boundseg}, and Proposition~\ref{prop I eta1 T} we easily conclude Theorem~\ref{pro:wcp}.
\end{proof}
\begin{proof}[Proof of Proposition~\ref{prop reduction}]
    We split each factor on both side of~\eqref{eq:impbound} as in~\eqref{fI}. For the mixed moments involving only factors of \(\cI_\epsilon\) we conclude the approximate equality from~\eqref{eq:c}. For the terms with at least one factor of \(\cE_\epsilon\) we use the high probability bounds from Lemma~\ref{lem:hpbound} together with~\eqref{Eest} to conclude equality.  
\end{proof}

\subsection{Proofs of Lemma~\ref{lem:boundseg} and Proposition~\ref{prop I eta1 T}}
\begin{proof}[Proof of Lemma~\ref{lem:boundseg}]  
    This argument was essentially given in~\cite[Lemmata 2-4]{1908.00969}, we repeat the proof here for completeness. We denote the eigenvalues of \(H^z\) by \(\set{\lambda_{\pm i}^z}_{i\in [n]}\) which are symmetric around \(0\) by block structure of \(H^z\) and, in modulus, agree with the singular values of \(X-z\). For notational simplicity we omit the \(z\)-dependence within the proof of Lemma~\ref{lem:boundseg}. 
    
    We start by splitting the \(\eta\)-integral in \(I_0^{\eta_0}\) as 
    \begin{equation}
        \label{eq:smallsplit}
        \begin{split}
            &  n\int_0^{\eta_0} \Im \braket{G^z(\ii\eta)-M^z(\ii\eta)} \diff \eta \\
            &\qquad =\sum_{\abs{\lambda_i}< n^{-l}}\log\left( 1+\frac{\eta_0^2}{\lambda_i^2}\right) + \sum_{\abs{\lambda_i}\ge n^{-l}}\log\left( 1+\frac{\eta_0^2}{\lambda_i^2}\right)-n\int_0^{\eta_0} \Im \widehat{m}^z(\ii\eta) \diff \eta,
        \end{split}
    \end{equation}
    where \(l\in\N\) is a large fixed positive integer, and \(\eta_0=n^{-1-\epsilon}\). For  Lemma~\ref{lem:boundseg} it is enough to prove that the rhs.\ of~\eqref{eq:smallsplit} is bounded by \(n^{-\epsilon/6+\xi}\) since \(\Delta f\) in \(I_0^{\eta_0}\) is bounded in \(L^1\). Using that \(\abs{m^z(\ii\eta)}\lesssim 1\), the third term in the second line of~\eqref{eq:smallsplit} is bounded by \(n\eta_0=n^{-\epsilon}\). 
    
    For the bounds on the first and second term in~\eqref{eq:smallsplit} we present two proofs; one relying on~\cite{TaoVusmooth}, and one relying on~\cite{MR3770875} under the additional mild moment assumption~\eqref{eq:reg}. For the first term in the rhs.\ of~\eqref{eq:smallsplit} we compute
    \begin{equation}\label{eq:estsmeig1}
        \begin{split}
            \E   \sum_{\abs{\lambda_i}< n^{-l}} \log\left( 1+\frac{\eta_0^2}{\lambda_i^2}\right) &\le n\E  \left[ \log\left( 1+\frac{\eta_0^2}{\lambda_1^2}\right) \1(\lambda_1\le n^{-l})\right]\\
            &\lesssim n\E [\abs{\log \lambda_1} \1(\lambda_1\le n^{-l})] \\
            &=n\int_{l\log n}\Prob(\lambda_1\le e^{-s})  \diff s.
        \end{split}
    \end{equation}
    For discrete random variables \(\chi\) the event \(\lambda_1=0\) might occur with some small but non-zero probability. However, using~\eqref{eq:boundsmeg} and a grid argument in the \(z\)-variable we can guarantee that \(\lambda_1=\lambda_1^z\ge n^{-l}\) holds simultaneously for all \(z\) on a very high probability event \(\Sigma\) with \(\Prob(\Sigma^c)\lesssim n^{-100}\). For the second term on the rhs.\ of~\eqref{eq:smallsplit} we again use~\eqref{eq:boundsmeg} to conclude
    \begin{equation}\label{eq:estsmeig2}
        \E \abs*{ \set*{i\given\abs{\lambda_i}\le n^{\epsilon/2}\eta_0} }\lesssim n^{-\epsilon/4}
    \end{equation}
    and thereby, using~\eqref{eq:estsmeig2}, and \(\log (1+x)\le x\),
    \begin{equation}\label{eq:usefsplit}
        \begin{split}
            \E \sum_{\abs{\lambda_i}\ge n^{-l}}\log\left( 1+\frac{\eta_0^2}{\lambda_i^2}\right)&=  \E\sum_{ n^{-l}\le \abs{\lambda_i}\le n^{\epsilon/2} \eta_0}\log\left( 1+\frac{\eta_0^2}{\lambda_i^2}\right)+
            \E  \eta_0^2 \sum_{\abs{\lambda_i}\ge n^{\epsilon/2}\eta_0} \frac{1}{\lambda_i^2} \\
            &\lesssim \E\abs{\set{i\given\abs{\lambda_i}< n^{\epsilon/2}\eta_0}} \cdot \log n + \E\eta_0^2 \sum_{\abs{\lambda_i}\ge n^{\epsilon/2}\eta_0} \frac{1}{\lambda_i^2} \\
            &\lesssim (\log n) n^{-\epsilon/4} +\E \eta_0\sum_{\abs{\lambda_i}\ge n^{\epsilon/2}\eta_0} 
            \frac{n^\epsilon\eta_0}{\lambda_i^2+(n^\epsilon\eta_0)^2} \\
            &\lesssim  (\log n) n^{-\epsilon/4}+ n\eta_0 \braket{\Im G^z(\ii n^\epsilon\eta_0)} \le    n^{\xi-\epsilon/4},
        \end{split}
    \end{equation}
    where in the last inequality we used averaged local law in~\eqref{ave}. By combining~\eqref{eq:estsmeig1} and~\eqref{eq:usefsplit} we conclude the claimed bound on \(\abs{I_0^{\eta_0}}\) conditionally on the high-probability event \(\Sigma\). However, by a trivial cut-off argument due to~\eqref{eq:boundmom} we may assume that \(\chi\) is bounded by \(\abs{\chi}\le n\). Then the lhs.\ of~\eqref{eq:bigdiv}, and the integrals \(J_T\), \(I_{\eta_0}^{\eta_1}\), \(I_{\eta_1}^T\), \(I_T^\infty\) in~\eqref{eq:bigdiv} are bounded deterministically by, say, \(n^2\), hence so is \(I_0^{\eta_0}\), and we conclude the claimed bound on \(\abs{I_0^{\eta_0}}\) also unconditionally.

    We may also complete the proof without relying on~\cite{TaoVusmooth}.  Under the additonal regularlity assumption~\eqref{eq:reg}, due to~\cite[Proposition 5.7]{MR3770875} it follows that
    \begin{equation}\label{eq:b}
        \Prob\left( \lambda_1\le \frac{u}{n} \right)\lesssim u^\frac{2\alpha}{1+\alpha}n^{\beta+1},
    \end{equation}
    for any \(u>0\). Then~\eqref{eq:b} allows us to estimate the rhs.\ of~\eqref{eq:estsmeig1} by, say, \(n^{-10}\) by choosing \(l\) large enough.
    The proof of~\eqref{eq:estsmeig2} can even avoid the smoothness assumption~\eqref{eq:reg}, using~\cite[Eq. (4a)]{1908.01653} (see below~\eqref{eq:smass} for more details). 
\end{proof}

\begin{proof}[Proof of Proposition~\ref{prop I eta1 T}]
    In order to estimate
    \[\E\abs{I_{\eta_1}^T}^2 = n^2 \int_{\C^2}\diff^2 z\diff^2 z' \Delta f(z)\ov{\Delta f(z')} \int_{[\eta_1,T]^2}\diff \eta\diff\eta' \braket{\Im G^z(\ii\eta)-\Im m^z(\ii\eta)}\braket{\Im G^{z'}(\ii\eta')-\Im m^{z'}(\ii\eta')} \]
    we perform a first-order Taylor expansion of \(z'\) around \(z\),
    \begin{equation}\label{G - m Taylor}
        \begin{split}
            \braket{\Im G^{z'}(\ii\eta')-\Im m^{z'}(\ii\eta')}&=\braket{\Im G^z(\ii\eta')-\Im m^{z}(\ii\eta')}\\
            &\quad +\int_0^1\Im\Bigl[\partial_z\braket{G^z(\ii\eta')-m^z(\ii\eta')}\vert_{z=z(s)}\Bigr] (z'-z)\diff s\\
            &\quad +\int_0^1 \Im\Bigl[\partial_{\overline{z}}\braket{G^z(\ii\eta')-m^z(\ii\eta')}\vert_{z=z(s)}\Bigr] (\overline{z'-z}) \diff s,
        \end{split}
    \end{equation}
    where \(z(s)=sz'+(1-s)z\). For the derivatives we have the bounds (see~\cite{NonhermitianGGG})
    \begin{equation}\label{GFG local law}
        \abs{\partial_{\ov{z}}\braket{G^z -  m^z}} + \abs{\partial_{z}\braket{G^z -  m^z}} \lesssim  \frac{n^\xi}{n\eta^{3/2}}
    \end{equation}

    and from using~\eqref{G - m Taylor} and~\eqref{GFG local law} we thus conclude
    \[ \begin{split}
        \E\abs{I_{\eta_1}^T}^2 &= n^2 \int_{\C^2}\diff^2 z\diff^2 z' \Delta f(z)\ov{\Delta f(z')} \int_{[\eta_1,T]^2}\diff \eta\diff\eta' \landauO*{n^\xi\frac{1}{n\eta}\frac{\abs{z-z'}}{n(\eta')^{3/2}}} \\
        &\lesssim n^\xi n^2 \norm{\Delta f}_1^2 \frac{\log n}{n} \frac{1}{n^{3/2}\eta_1^{1/2}}  \lesssim n^{2\xi} \frac{\norm{\Delta f}_1^2}{\sqrt{n\eta_1}}=n^{2\xi-\epsilon/2}\norm{\Delta f}_1^2,
    \end{split} \]
    since the \(0\)-th term of~\eqref{G - m Taylor} does not contribute to the integral due to \(\int_\C \Delta f(z)\diff^2 z =0\). 
\end{proof}

\printbibliography

\end{document}

%% file: packages.tex
\usepackage[T1]{fontenc} 
\usepackage[english]{babel}
\usepackage[p,osf]{cochineal}
\usepackage[varqu,varl,var0]{inconsolata}
\usepackage[scale=.95,type1]{cabin} 
\usepackage{mathalfa}
\usepackage[utf8]{inputenc} 
\usepackage{amsthm}
\usepackage{amsmath}
\usepackage{stackengine}
\usepackage{amssymb}
\usepackage{amsfonts}
\usepackage{amsaddr}
\usepackage{xspace}
\usepackage{enumitem} 
\usepackage{csquotes}
\usepackage{xcolor}
\usepackage[colorlinks]{hyperref}
\definecolor{intcolor}{HTML}{CA0020}
\definecolor{extcolor}{HTML}{0571B0}
\hypersetup{breaklinks,linkcolor=intcolor,citecolor=intcolor,filecolor=black,urlcolor=extcolor,menucolor=intcolor,runcolor=extcolor}
\usepackage{graphicx}
\usepackage{mathtools}
\mathtoolsset{centercolon}
\usepackage{bm}

%% file: notations.tex
\DeclareMathOperator{\Tr}{Tr}

\DeclareMathOperator{\E}{\mathbf{E}}
\DeclareMathOperator{\Prob}{\mathbf{P}}

\DeclareMathOperator{\Spec}{Spec}
\DeclareMathOperator{\erf}{erf}
\newcommand{\ov}{\overline}
\newcommand{\ii}{\mathrm{i}}

\ifdefined\C
\renewcommand{\C}{\mathbf{C}}
\else
\newcommand{\C}{\mathbf{C}}
\fi
\newcommand{\HC}{\mathbf{H}}

\newcommand{\vx}{\bm{x}}
\newcommand{\cI}{\mathcal{I}}
\newcommand{\vw}{\bm{w}}
\newcommand{\vz}{\bm{z}}
\newcommand{\vy}{\bm{y}}
\newcommand{\Gin}{\mathrm{Gin}}

\newcommand{\wt}{\widetilde}

\newcommand{\cE}{\mathcal{E}}
\newcommand{\R}{\mathbf{R}}
\newcommand{\F}{\mathbf{F}}
\newcommand{\1}{\bm{1}}
\newcommand{\N}{\mathbf{N}}
\newcommand{\Z}{\mathbf{Z}}

\newcommand{\DD}{\mathbf{D}}

\newcommand{\cO}{\mathcal{O}}
\newcommand{\co}{{\scriptstyle\mathcal{O}}}

\newcommand{\diff}{\operatorname{d}\!{}}
\DeclarePairedDelimiter{\braket}{\langle}{\rangle}%
\DeclarePairedDelimiter{\abs}{\lvert}{\rvert}%
\DeclarePairedDelimiter{\norm}{\lVert}{\rVert}%
\providecommand\given{}
\newcommand\SetSymbol[1][]{\nonscript\:#1\vert\allowbreak\nonscript\:\mathopen{}}
\DeclarePairedDelimiterX{\tuple}[1](){\renewcommand\given{\SetSymbol[\delimsize]}#1}
\DeclarePairedDelimiterX{\set}[1]{\{}{\}}{\renewcommand\given{\SetSymbol[\delimsize]}#1}
\DeclarePairedDelimiterX{\Set}[1]\{\}{\renewcommand\given{\SetSymbol[\delimsize]}#1}
\DeclarePairedDelimiterXPP{\landauO}[1]{\cO}(){}{#1}
\DeclarePairedDelimiterXPP{\landauo}[1]{\co}(){}{#1}
\DeclarePairedDelimiterXPP{\landauok}[1]{\co_k}(){}{#1}
\DeclarePairedDelimiterXPP{\landauOprec}[1]{\cO_\prec}(){}{#1}
\DeclarePairedDelimiterXPP{\Exp}[1]{\E}[]{}{\renewcommand\given{\SetSymbol[\delimsize]}#1}
\DeclareFontFamily{U}{mathx}{\hyphenchar\font45}
\DeclareFontShape{U}{mathx}{m}{n}{
      <5> <6> <7> <8> <9> <10>
      <10.95> <12> <14.4> <17.28> <20.74> <24.88>
      mathx10
      }{}
\DeclareSymbolFont{mathx}{U}{mathx}{m}{n}
\DeclareFontSubstitution{U}{mathx}{m}{n}
\DeclareMathAccent{\widecheck}{0}{mathx}{"71}

%% file: bibliography_setup.tex
\usepackage[giveninits=true,url=false,doi=false,isbn=false,eprint=true,datamodel=mrnumber,sorting=nty,maxcitenames=4,maxbibnames=99,backref=false,block=space,backend=biber,bibstyle=phys]{biblatex} 
\AtEveryBibitem{\clearfield{month}}
\AtEveryCitekey{\clearfield{month}} 
\renewbibmacro{in:}{}
\DeclareFieldFormat[article]{title}{\emph{#1}} 
\DeclareFieldFormat{mrnumber}{\ifhyperref{\href{http://www.ams.org/mathscinet-getitem?mr=#1}{\nolinkurl{MR#1}}}{\nolinkurl{#1}}}
\DeclareFieldFormat{pmid}{\ifhyperref{\href{https://www.ncbi.nlm.nih.gov/pubmed/#1}{\nolinkurl{PMID#1}}}{\nolinkurl{#1}}}
\DeclareFieldFormat{eprint}{\ifhyperref{\href{https://arxiv.org/abs/#1}{\nolinkurl{arXiv:#1}}}{\nolinkurl{#1}}}
\renewbibmacro*{doi+eprint+url}{%
  \iftoggle{bbx:doi}{\printfield{doi}}{}
  \newunit\newblock%
  \printfield{mrnumber}%
  \newunit\newblock%
  \printfield{pmid}%
  \newunit\newblock%
  \printfield{eprint}%
  \iftoggle{bbx:url}{\usebibmacro{url+urldate}}{}}